\documentclass[12pt]{amsproc}
\usepackage[pdftex]{hyperref}
\usepackage{color}
\newtheorem{theorem}{\sc Theorem}[section]
\newtheorem{lemma}[theorem]{\sc Lemma}
\newtheorem{proposition}[theorem]{\sc Proposition}

\newtheorem{remark}[theorem]{\sc Remark}
\begin{document}

\title[Varieties and the problem on conciseness of words]{Varieties of groups and the problem on conciseness of words}

\author{Cristina Acciarri}

\address{Cristina Acciarri: Dipartimento di Scienze Fisiche, Informatiche e Matematiche, Universit\`a degli Studi di Modena e Reggio Emilia, Via Campi 213/b, I-41125 Modena, Italy}
\email{cristina.acciarri@unimore.it}

\author{Pavel Shumyatsky }
\address{ Pavel Shumyatsky: Department of Mathematics, University of Brasilia,
Brasilia-DF, 70910-900 Brazil}
\email{pavel@unb.br}

\thanks{Part of this work was done while the second author was visiting the Department of Mathematics of the University of the Basque Country. He expresses his sincere gratitude to the department for the excellent hospitality.  The first author is member of ``National Group for Algebraic and Geometric Structures, and Their Applications'' (GNSAGA–INdAM). The second author was partially supported by CNPq. The authors thank the referee for their helpful comments leading to improvements of an earlier version of the paper.}

\keywords{Residually finite groups, concise words, Lie methods, varieties of groups}
\subjclass[2020]{20E26, 20F10, 20F40, 20E10, 20F45}

\begin{abstract} 
A group-word $w$ is concise in a class of groups $\mathcal X$ if and only if the verbal subgroup $w(G)$ is finite whenever $w$ takes only finitely many values in a group $G\in \mathcal X$. It is a long-standing open problem whether every word is concise in residually finite groups. In this paper we observe that the conciseness of a word $w$ in residually finite groups is equivalent to that in the class of virtually pro-$p$ groups.

This is used to show that if $q,n$ are positive integers and $w$ is a multilinear commutator word, then the words $w^q$ and $[w^q,_{n} y]$ are concise in residually finite groups. Earlier this was known only in the case where $q$ is a prime power.

In the course of the proof we establish that certain classes of groups satisfying the law $w^q\equiv1$, or $[\delta_k^q,{}_n\, y]\equiv1$, are varieties.

\end{abstract}

\maketitle

\section{Introduction}

Let $w=w(x_1,\dots,x_k)$ be a group-word. Given a group $G$, we denote by $w(G)$ the verbal subgroup corresponding to the word $w$, that is, the subgroup generated by the set of all values $w(g_1,\ldots,g_k)$, where $g_1,\ldots,g_k$ are elements of $G$. The word $w$ is concise if $w(G)$ is finite whenever the set of $w$-values in $G$ is finite. More generally, the word $w$ is called concise in a class of groups $\mathcal X$ if $w(G)$ is finite whenever the set of $w$-values in $G$ is finite for a group $G\in\mathcal X$. In the sixties Hall raised the problem whether all  words are concise, but in 1989  Ivanov \cite{ivanov} (see also \cite[p.\ 439]{ols}) solved the problem in the negative. On the other hand, the problem for residually finite groups remains open. In recent years several new positive results with respect to this problem were obtained (see \cite{AS14,GS2015,dms-19,DMS2019,DMS2020}).

Recall that  multilinear commutator words, also known under the name of outer commutator words, are precisely the words that can be written in the form of multilinear Lie monomials. Important examples of  multilinear commutator words are provided by the lower central words  $\gamma_{k}$ on $k$ variables  defined inductively by the formulae
\[
\gamma_1=x_1,
\qquad
\gamma_k=[\gamma_{k-1}(x_1,\ldots,x_{k-1}),x_k]=[x_1,\ldots,x_k],
\quad
\text{for $k\ge 2$.}
\]
The corresponding verbal subgroups $\gamma_k(G)$  are  the terms of the lower central series of $G$. 
Another distinguished series of multilinear commutator words are the derived words $\delta_k$ on $2^k$  variables defined by the formulae
\[
\delta_0=x_1,
\quad
\delta_k=[\delta_{k-1}(x_1,\ldots,x_{2^{k-1}}),\delta_{k-1}(x_{2^{k-1}+1},\ldots,x_{2^{k}})]
\quad
\text{for $k\geq 1 $.}
\]
 The corresponding verbal subgroups  $\delta_k(G)$ are  the  terms $G^{(k)}$ of the derived series of $G$. 

The word $[x,{}_{n}\, y]$ is defined inductively by $[x,_0y]=x$ and $[x,{}_{n}\,y]=[[x,{}_{n-1}\,y],y]$ for $n\geq1$.

Recall that a group has a property virtually if it has a subgroup of finite index with that property.
In the present paper we observe that conciseness of a word $w$ in residually finite groups is equivalent to that in the class of virtually pro-$p$ groups (see Proposition \ref{start}). This enables us to establish the following theorems.

\begin{theorem}\label{main1}
Suppose that $w$ is a multilinear commutator word. For any integer $q\geq 1$ the word $w^q$ is concise in residually finite groups.
\end{theorem}

 \begin{theorem}\label{main2}
Suppose that $w=w(x_1,\ldots,x_k)$ is a multilinear commutator word. For any integers $q\geq 1$ and $n\geq 0$ the word $[w^q,{}_n\,y]$ is concise in residually finite groups.
\end{theorem}

The particular cases of the above theorems in which $q$ is a prime power were previously established in \cite{AS14} and \cite{DMS2020}, respectively. The proofs given in this paper are based on new results on varieties of groups.

A variety is a class of groups defined by equations. More precisely, if $W$ is a set of words, the class of all groups satisfying the laws $W\equiv 1$ is called the variety determined by $W$. By a well-known theorem of Birkhoff \cite[2.3.5]{Rob}, varieties are precisely classes of groups closed with respect to taking subgroups, quotients and Cartesian products of their members. Some interesting varieties of groups have been discovered in the context of the 
restricted Burnside problem, solved in the affirmative by Zelmanov \cite{RBP1,RBP2}. In particular, we now know that the class of locally finite groups satisfying the law $x^q\equiv1$ is a variety.

Recall that a group is said to locally have some property if all its finitely generated subgroups have that property. A number of varieties of (locally nilpotent)-by-soluble groups were presented in \cite{shum2002, shu1,STT1}. Here we will prove the following results.

\begin{theorem}\label{variety1} For positive integers $j,q$ and a multilinear commutator word $w$, let $\frak X= \frak X(j,q,w)$ be the class of all groups $G$ such that\begin{enumerate}
\item  $G$ satisfies the law $w^q\equiv1$;
\item  The verbal subgroup $w(G)$ is locally finite;
\item  Locally soluble subgroups of $w(G)$ have Fitting height at most~$j$. \end{enumerate}
Then $\frak X$ is a variety.
\end{theorem}

\begin{theorem}\label{variety2} For any integers $j,q\geq 1$ and $k,n\geq 0$, let $\frak Y= \frak Y(j,k,n,q)$ be the class of all groups $G$ such that

\begin{enumerate}
\item  $G$ satisfies the law $[\delta_k^q,{}_n\,y]\equiv1$;
\item  The commutator subgroup $G^{(k)}$ is (locally nilpotent)-by-(locally finite);
\item  Locally soluble subgroups of $G^{(k)}$ have Fitting height at most $j$. \end{enumerate}
Then $\frak Y$ is a variety.
\end{theorem}

It is unclear whether the word $\delta_k$ in the above theorem can be replaced by an arbitrary multilinear commutator word. Another interesting question is whether the classes of groups obtained by omitting the assumption $(3)$ in Theorems \ref{variety1} and \ref{variety2} are varieties, too.

The notation used in this paper is mostly standard. As usual, the group generated by a set $M$ will be denoted by $\langle M\rangle$, the commutator subgroup of a group $G$ by  the symbol $G'$ and  the Frattini subgroup of $G$ by $\Phi(G)$. We say that a quantity is $(a,b,\dots)$-bounded if it is bounded by a number depending only on the parameters $a,b,\dots$.



\section{Preliminaries}
In this section we collect some results that will be useful in the next sections to establish the main theorems. 

A proof of the following well-known result can be found in \cite[Lemma 4.5]{DMS2022}
\begin{lemma}\label{frattini_finito} Let $N$ be a normal subgroup of a finite group $G$. Then there exists a subgroup $H$ of $G$ such that $G=HN$ and $H\cap N\leq \Phi(H)$.
\end{lemma}

We will also require the following lemmas from \cite{shum2000}.
\begin{lemma}\cite[Lemma 4.1]{shum2000} \label{shum2000_lemma41} Let $G$ be a group and $w$ a multilinear commutator word on $k$ variables. Then every $\delta_k$-value in $G$ is a $w$-value.
\end{lemma}

\begin{lemma}\cite[Lemma 4.2]{shum2000}\label{shum2000_lemma42} Let $w$ be a multilinear commutator word and $G$ a soluble group in which all $w$-values have finite order. Then the verbal subgroup $w(G)$ is locally finite.
\end{lemma}

An element $g$ of a group $G$ is called a (left) Engel element if for any $x\in G$ there exists $n=n(g,x)\geq 1$ such that $[x,{}_n\,g]=1$. If $n$ can be chosen independently of $x$, then $g$ is a (left) $n$-Engel element.
In a similar way an element $g\in G$ is said to be a right Engel element if for each $x\in G$ there exists a positive integer $n=n(g,x)$ such that $[g,{}_n\, x]=1$. If $n$ does not depend on  $x$, then $g$ is a right $n$-Engel element. The next observation is due to Heineken (see \cite[12.3.1]{Rob}).
\begin{lemma}\label{engel_LR}
Let $g$ be a right $n$-Engel element in a group $G$. Then $g^{-1}$ is a left $(n+1)$-Engel element.
\end{lemma}

The following theorem is a well-known result of  Baer \cite{Baer} (see also \cite[Satz III.6.15]{Hupp}).
\begin{theorem}\label{engel_in_fitting}
An Engel element of a finite group belongs to the Fitting subgroup.
\end{theorem}
A proof of the following result can be found in \cite[Lemma 9]{DMS2019}.

\begin{lemma}\label{dms_lemma1} Let $G=U\langle t\rangle$ be a group that is a product of a normal subgroup $U$ and a cyclic subgroup $\langle t\rangle$. Assume that $U$ is nilpotent of class $c$ and there exists a generating set $Y$ of $U$ such that $[y,{}_n\,t]=1$ for every $y\in Y$. Then  $G$ is nilpotent of $(c,n)$-bounded class.
\end{lemma}

 As usual, if $\alpha$ is an automorphism of a group $G$, we write $[G,\alpha]$ for the subgroup generated by all elements of the form $g^{-1}g^{\alpha}$, where $g\in G$. By induction we define $[G,{}_n \alpha]=[[G,{}_{n-1} \alpha],\alpha]$. The following lemma is due to Casolo.

\begin{lemma}\cite[Lemma 6]{C}\label{dms_lemma2} 
Let $A$ be an abelian group, and let $x$ be an automorphism of $A$ such that $[A,{}_n\,x]=1$ for some $n\geq 1$. If $x$ has finite order $q$, then $[A^{q^{n-1}},x]=1$.
\end{lemma}


\section{On the Fitting height of a group}
For a group $G$ let $F(G)$ denote the Hirsch–Plotkin radical of $G$, which is the largest locally nilpotent normal subgroup of $G$. We define $F_0(G)=1$ and let $F_{i+1}(G)$ be the full inverse image of $F(G/F_i(G))$ for $i=0,1,\ldots$.  The group  $G$ is said to be of finite Fitting height $j$ if $G=F_j(G)$ for some integer $j$ and $j$ is the least such integer. In this case we write $h(G)=j$. Observe that if $G$ is a finite group, then the Hirsch–Plotkin radical $F(G)$ is just the Fitting subgroup of $G$.

The coprime commutators  $\delta_k^*$ were introduced in \cite{forum}.  For the  reader's convenience we recall here the definition. Let $G$ be a finite group. Every element of $G$ is  a $\delta_0^*$-commutator.  For $k\geq 1$ let $Y$ be the set of all elements of $G$ that are powers of $\delta_{k-1}^*$-commutators. The element $g$ is a $\delta_k^*$-commutator if there exist $a,b\in Y$ such that $g=[a,b]$ and $(|a|,|b|)=1$. The subgroup of $G$ generated by  all  $\delta_k^*$-commutators will be denoted by $\delta_k^*(G)$. One can easily see that if $N$ is a normal subgroup of $G$ and $x$ is a $\delta_k^*$-commutator, then $xN$ is a $\delta_k^*$-commutator in $G/N$. Moreover if $y$ is an element whose image in $G/N$ is a $\delta_k^*$-commutator, then there exists a $\delta_k^*$-commutator $z$ in $G$ such that $y\in zN$.  It is also clear from the definition that if $x$ is a $\delta_k^*$-commutator, then there are $x_1,\dots,x_{2^k}\in G$ and a word $\omega$ obtained from $\delta_k$ adding some powers to sub-commutators such that $x=\omega(x_1,\dots,x_{2^k})$.

It was shown in \cite{forum} that  $\delta_k^*(G)=1$ if and only if the Fitting height of $G$ is at most $k$. It follows that for every $k\geq1$ the subgroup $\delta_k^*(G)$ is precisely the last term of the lower central series of $\delta_{k-1}^*(G)$, that is, $\delta_k^*(G)=\gamma_\infty(\delta_{k-1}^*(G))$. Observe that the definition of $\delta_k^*$-commutators naturally extends to locally finite groups.

In this section we establish some results on the Fitting height and other related length parameters that will  be useful throughout  the paper.
\begin{lemma}\label{loc_finite_finit_generat} 
Let $j$ be a positive integer and $G$ a locally finite group such that $h(K)\leq j$  whenever $K$ is a finite soluble subgroup of $G$. Then $h(H)\leq j$ for every locally soluble subgroup $H$ of $G$.
\end{lemma}
\begin{proof} 
Without less of generality we can assume that $G$ is locally soluble. It suffices to show that $G$ has a normal series of length at most $j$ such that each factor of the series is a locally nilpotent group.

We argue by induction of $j$. If $j=1$, then $G$ is locally nilpotent and there is nothing to prove. 
Assume that $j\geq 2$. For $i=1,2\dots$ let $X_i^*$ be the set of $\delta_i^*$-commutator in $G$. Observe that $X_j^*=1$. Indeed, for any $y\in X_j^* $, there are elements $b_1,\dots,b_{2^{j}}\in G$  and a word $\omega$ such that  $y=\omega( b_1,\dots, b_{2^{j}}).$ Set  $K=\langle b_1,\dots,b_{2^{j}}\rangle$. Note that $y$ is a $\delta_j^*$-commutator in a finite soluble subgroup $K\leq G$. Since $h(K)\leq j$, we  know that $\delta_j^*(K)=1$ and so $y$ is trivial. This implies that   $X_j^*=1$, as desired. 

 Let $F=F(G)$ the Hirsch-Plotkin radical of $G$. In order  to prove that $h(G)\leq j$,  it  would be sufficient to show that $X_{j-1}^*\subseteq F$. In turn, this will be deduced from the fact that every element in $X_{j-1}^*$ is Engel in $G$. 

Let $x\in X_{j-1}^*$. Similarly to what was done above, there exist elements $x_1,\dots,x_{2^{j-1}}\in G$  and a word $\omega_1$ such that  $x=\omega_1 (x_1,\dots, x_{2^{j-1}}).$ Let $g\in G$ and  set $M=\langle g,x_1,\dots,x_{2^{j-1}}\rangle$. Since $h(M)\leq j$, any  $\delta_j^*$-commutator in $M$ is trivial and so, in particular, $\delta_{j-1}^*(M)$ is contained in the Fitting subgroup of $M$. Thus $x\in F(M)$ and there exists $n=n(g,x)$ such that $[g,_n x]=1$, that is, $x$ is an Engel element in $G$.  It follows from  Theorem \ref{engel_in_fitting} that $x\in F$ and so $X^*_{j-1}\subseteq F$, as desired. Set $\bar{G}=G/\delta^*_{j-1}(G)$. By the inductive hypothesis we know that $h(\bar{G})\leq j-1$, and so we conclude that $h(G)\leq j$. This concludes the proof.
\end{proof}

Recall that the generalized Fitting subgroup $F^*(G)$ of a finite group $G$ is  the product of the Fitting subgroup $F(G)$ and all subnormal quasisimple subgroups; here a group is quasisimple if it is perfect and its quotient by the centre is a
non-abelian simple group. The generalized Fitting series of $G$ is defined by
$F_0^*(G) =1$ and, by induction, $F_{i+1}^*(G)$ being the inverse image of $F^*(G/F^*_i(G))$. The least number $h$ such that $F^*(G)=G$ is the generalized Fitting height $h^*(G)$ of $G$. Clearly, if $G$ is soluble, then $h^*(G)=h(G)$ is the ordinary Fitting height of $G$.

\begin{lemma}\label{classY}
Let $j$ be a positive integer and $G$ a finite group such that $h(K)\leq j$ whenever $K$ is a soluble subgroup of $G$.  Then
\begin{itemize} 
\item[(i)]  The hypothesis is inherited by homomorphic images of $G$;
\item[(ii)] $h^*(G)$ is $j$-bounded.
\end{itemize}
\end{lemma}
\begin{proof}
To prove (i) let $N$ be a normal subgroup of $G$. For any soluble  subgroup $KN/N$ of $G/N$ we have to show that  $h(KN/N)\leq j$. Indeed, by Lemma \ref{frattini_finito} we can choose  $L\leq KN$ such that $KN=LN$ and $L\cap N\leq \Phi(L)$. Note that $L$ is soluble because $L/\Phi(L)$ is, being isomorphic to a homomorphic image of  $KN/N$. Therefore $h(L)$ is at most $j$. Since $KN/N$ is a homomorphic image of $L$, we conclude that $h(KN/N)\leq j$.

Claim (ii) is immediate from  \cite[Theorem 1.3]{DS}.\end{proof}

Remark that it is an open question (see \cite[Problem 1.4]{DS}) whether $h^*(G)$ is at most $j$ whenever $G$ is as in the above lemma.

\begin{lemma}\label{enough_fg}
Let $j$ be a positive integer and  $G$  a (locally nilpotent)-by-(locally finite) group. Assume that $G$ is locally soluble and for any finitely generated subgroup $H\leq G$ we have $h(H)\leq j$. Then $h(G)\leq j$. 
\end{lemma}
\begin{proof}

Let $F=F(G)$ the Hirsch-Plotkin radical of $G$. For $i=0,1,2,\dots$ set $$G_i^*=\{x\in G;\ xF\text{ is a $\delta_i^*$-commutator in $G/F$}\}.$$
To prove that $h(G)\leq j$,  it is sufficient to show that $G_{j-1}^*= F$. This will be deduced from Theorem \ref{engel_in_fitting} and the fact that every element in $G_{j-1}^*$ is Engel in $G$. In what follows we write $\bar{G}$ for $G/F$.

Let $x\in G_{j-1}^*$. There are elements $x_1,\dots,x_{2^{j-1}}\in G$  and a word $\omega$, obtained from the word $\delta_{j-1}$ adding some powers to subcommutators, such that  $$\bar{x}=\omega(\bar x_1,\dots,\bar x_{2^{j-1}}).$$ Let $y\in G$ and $H=\langle y,x_1,\dots,x_{2^{j-1}}\rangle$. By hypothesis, $h(H)\leq j$. Let $h(H)=k$ and consider the series
$$H=H_1\geq H_2\geq\dots\geq H_k\geq1,$$ where $H_k=F(H)$ and $H_i/H_{i+1}=\delta_i^*(H/H_{i+1})$ for $i=1,\dots,k-1$.  Note that $H_k$ has finite index in $H$, because $H$ is finitely generated (locally nilpotent)-by-(locally finite) and so $H/H_k$ is finite, being a finitely generated locally finite group. 
Thus  $h(H/H_k)=k-1\leq j-1$ and we deduce that $\delta_{j-1}^*(H/H_k)=1$. In particular, it follows that $x\in F(H)$. Since $y\in H$,  there exists a positive integer $n$ such that $[y,{}_{n}\,x]=1$ and so $x$ is an Engel element in $G$, as desired. Since $x$ is arbitrary in $G_{j-1}^*$, we have $G_{j-1}^*= F$.  Hence, the quotient $\bar{G}$ is locally finite and  $\delta_{j-1}^*(\bar{G})=1$. This implies that $h(\bar{G})\leq j-1$ and so $h(G)\leq j$, as desired.
\end{proof}

\begin{proposition}\label{closed_quotients}
Let $j\geq 1$ and $G$ be a (locally nilpotent)-by-(locally finite)  group such that $h(K)\leq j$ for any locally soluble subgroup $K\leq G$. Then the hypothesis is inherited by homomorphic images of $G$.
\end{proposition}
\begin{proof} 
Let $N$ be a normal subgroup of $G$. Without loss of generality we may assume that $G/N$ is locally soluble. We need to show that $h(G/N)\leq j$. 

Assume by contradiction that this is false. In view of Lemma \ref{enough_fg} there is a finitely generated soluble subgroup of $G/N$ with Fitting height at least $j+1$, so we can simply assume that $G$ itself is finitely generated.  In particular $G/N$ is soluble. Since $G$ is (locally nilpotent)-by-(locally finite), it is also virtually nilpotent. Let $T$ be a normal nilpotent subgroup of finite index in $G$. Set $\bar{G}=G/T$. In view of Lemma \ref{frattini_finito}, there exists a soluble subgroup $\bar S\leq\bar G$ such that $\bar G=\bar N\bar S$.  Let $S$ be the inverse image of $\bar S$. It is clear that $S$ is soluble and $G=NS$. Note that $G/N$ is isomorphic to $S/(S\cap N)$ and so $h(G/N)\leq h(S)\leq j$, as desired. This concludes the proof. 
\end{proof}



\section{Bounding the order of a finite group}\label{Lie_machinery}

The purpose of this section is to establish the following lemma which plays a key role in our subsequent arguments. We say that a subset $X$ of a group $ G$ is normal if it is invariant under the inner automorphisms of $G$. The set is commutator-closed if $[x,y]\in X$ for any $x,y\in X$.

\begin{lemma}\label{basis} Let $q,s$ be positive integers and $G$ a finite group generated by a normal commutator-closed set $X$ such that $x^q=1$ and $[g,x_1,\dots,x_s]\in X$ for any $g\in G$ and $x,x_1,\dots,x_s\in X$. Suppose that $G$ satisfies a law $u\equiv1$ and $h^*(G)=h$. If $G$ is $m$-generator, then the order of $G$ is $(h,m,q,s,u)$-bounded.
\end{lemma}

In the sequel the above result will be used in the case where $X$ is the set of $\delta_k$-values of a group $G$, for some fixed $k\geq 1$. We will implicitly use the following observation.

\begin{remark} Let $k\geq 1$ and denote by $X$ the set of $\delta_k$-values in a group $G$. Then $[g,x_1,\ldots,x_k]\in X$, for any $g\in G$ and $x_1,\ldots,x_k \in X$.
\end{remark}
The proof of  Lemma \ref{basis} uses Lie-theoretical tools in the spirit of the solution of the restricted Burnside problem \cite{RBP1,RBP2}.

Let $L$ be a Lie algebra over the field. 
If $X\subseteq L$ is a subset of $L$, by a commutator in elements of $X$ we mean any element of $L$ that can be obtained as a Lie product of elements of $X$ with some system of brackets. 
 
 A deep result of Zelmanov (announced in \cite[III(0.4)]{zel1}, a detailed proof was published in \cite{zel2}), which has numerous important applications to group theory, says that if a Lie algebra $L$ is PI and is generated by finitely many elements all commutators in which are ad-nilpotent, then $L$ is nilpotent. The following result was deduced in  \cite[Corollary of Theorem 4]{KS} from Zelmanov's theorem.

\begin{theorem}\label{jpaa_thm24} Let $L$ be a Lie algebra over $\mathbb{F}_p$ generated by $m$ elements $a_1,\ldots,a_m$. Assume that $L$ satisfies the identity $f\equiv 0$ and that each monomial in the generators $a_1,\ldots,a_m$ is ad-nilpotent of index at most $n$. Then $L$ is nilpotent of $\{f,m,n\}$-bounded class.
\end{theorem}

Let $p$ be a prime. Given a group $G$, write $L_p(G)$ for the Lie algebra associated with the $p$-dimension series of $G$ (see \cite{shum} for details). If $x\in G$, the corresponding homogeneous element of $L_p(G)$ is denoted by $\tilde{x}$.

\begin{lemma}\label{jpaa_lemma22} \cite[p.\ 131]{lazard} For any $x\in G$ we have $(ad\tilde{x})^p=ad(\widetilde{x^p})$. Consequently, if $x^{p^s}=1$ then $\tilde{x}$ is ad-nilpotent of index at most $p^{s}$.
\end{lemma}

If a group-law $w\equiv 1$ holds in $G$, then $L_p(G)$ satisfies a polynomial identity whose form is determined by $p$ and $w$.

The next proposition follows from the proof of \cite[Theorem 1]{wz}.

\begin{proposition}\label{jpaa_lemma23}\cite[Lemma 2.3]{shum2002} Let $G$ be a group in which a law $w \equiv1$ holds. Then there exists a non-zero multilinear Lie polynomial $f$  over the field with $p$ elements, whose form depends only on $p$ and $w$, such that $L_p(G)$ satisfies the identity $f\equiv 0$.
\end{proposition}

Nilpotency of the algebra $L_p(G)$ usually enables one to draw strong conclusions about the group $G$. A proof of the next lemma can be found in \cite[Proposition 2.11]{shum}.

\begin{lemma}\label{jpaa_lemma21} Let a group $G$ be generated by elements $a_1,\dots,a_m$. Assume that $L_p(G)$ is nilpotent of class $c$ and let $\rho_1,\dots,\rho_s$ be the list of all possible commutators in $a_1,\dots,a_m$ on at most $c$ variables. Then the group $G$ can be written as a product $G=\langle \rho_1\rangle\langle \rho_2\rangle \cdots \langle \rho_s\rangle$ of the cyclic subgroups generated by $\rho_i$.
\end{lemma}

We are now ready to prove Lemma \ref{basis}.

\begin{proof}[Proof of Lemma \ref{basis}] Let $j$ be the minimal number such that $G$ has a normal series $$G=G_1>\dots>G_j=1,$$ all of whose factors are either nilpotent or direct products of non-abelian simple groups. Obviously, $j\leq2h$ so we can use induction on $j$. Let $N=G_{j-1}$ be the last nontrivial term of the above series. By induction the order of the quotient group $G/N$ is $(h,m,q,s,u)$-bounded and so $N$ has bounded index as well. Therefore $N$ can be generated by a bounded number of elements. Say $N$  is $m_0$-generator, where $m_0$ is $(h,m,q,s,u)$-bounded. 

Suppose that $N=S_1\times\dots\times S_t$ is a direct product of non-abelian simple groups $S_i$. By a result of Jones \cite{Jones} there are only finitely many simple groups satisfying the law $u\equiv1$ and we deduce that there is a $u$-bounded constant $C$ such that $|S_i|\leq C$ for any $i=1,\dots,t$. Taking into account that $N$ is $m_0$-generator, observe that the number of subgroups of $N$ of index at most $C$ is $(m_0,C)$-bounded (see \cite[Theorem 7.2.9]{mhall}) and therefore $t$ is $(h,m,q,s,u)$-bounded. It follows that the order of $N$ is $(h,m,q,s,u)$-bounded and we are done.

We therefore assume that $N$ is nilpotent.

Let $M$ be the subgroup generated by all elements $[g,x_1,\dots,x_s]$, where $g\in N$ and $x_1,\dots,x_s\in X$. Note that $M$ is normal in $G$ and $N/M$ is contained in $Z_s(G/M)$, the $s$th term of the upper central series of $G/M$. An application of the Baer theorem \cite[Corollary 2, p. 113]{Rob2} yields that $\gamma_{s+1}(G/M)$ has bounded order. It follows that the  quotient group $(G/M)/\gamma_{s+1}(G/M)$ is of bounded order, being  $m$-generator, nilpotent of class at most $s$,  and generated by elements of order dividing $q$. Thus, we deduce that the index $[G:M]$ is bounded as well. Therefore $M$ can be generated by boundedly many, say $m_1$, elements. It is sufficient to show that the order of $M$ is $(h,m,q,s,u)$-bounded.

Since $M$ is nilpotent and generated by $M\cap X$, it is clear that any prime divisor of $|M|$ is a divisor of $q$. Hence, it is sufficient to bound the order of the Sylow $p$-subgroup of $M$ for any prime $p$ dividing $q$. Write $M=P\times O_{p'}(M)$, where $P$ is the Sylow $p$-subgroup of $M$. Let $a_1,a_2,\dots$ be all the elements of $X$ contained in $M$, and for any $i$ we write $a_i=b_iy_i$, where $b_i\in P$ and $y_i\in O_{p'}(M)$. Then $P=\langle b_1,b_2,\dots\rangle$. Since $P$ is an $m_1$-generator finite $p$-group, the Burnside Basis Theorem \cite[5.3.2]{Rob} shows that $P$ is generated by some $m_1$ elements in the list $b_1,b_2,\dots$. Without any loss of generality we will assume that $P=\langle b_1,b_2,\dots, b_{m_1}\rangle$. Let $p_0$ be the maximal power of $p$ dividing $q$. Clearly, the order of any $b_i$ divides $p_0$. Moreover, since $[b_i,b_j]$ is again in the list $b_1,b_2,\dots$, it follows that any commutator in $b_1,b_2,\dots$ has order dividing $p_0$. 

Let $L=L_p(P)$. We know that the the identity $u\equiv 1$ holds in $P$ so it follows from Lemma \ref{jpaa_lemma23} that there exists a non-zero Lie polynomial $f$ over the field with $p$ elements, whose form depends only on $u$ and $q$ (recall that $p$ is a prime divisor of $q$), such that the algebra $L_p(P)$ satisfies the identity $f\equiv 0$.

Consider an arbitrary Lie monomial $\sigma$ in the generators $\tilde b_1,\tilde b_2,\dots,\tilde b_{m_1}$ of $L_p(P)$ and let $\rho$ be the group commutator in $b_1,b_2,\dots,b_{m_1}$ having the same system of brackets as $\sigma$. The definition of $L_p(P)$ yields that either $\sigma=0$ or $\sigma=\tilde\rho$. We know that  any commutator in $b_1,b_2,\dots$ has order dividing $p_0$, so $\rho^{p_0}=1$. Thus Lemma \ref{jpaa_lemma22} implies that $\sigma$ is ad-nilpotent of index at most $p_0$. Theorem \ref{jpaa_thm24} now says that $L_p(P)$ is nilpotent of class depending only on $m_1,u$ and $p_0$. Combining this with Lemma \ref{jpaa_lemma21} we conclude that there exists a number $l$, bounded in terms of $m_1,u,p_0$ only, such that $P$ can be written as a product of at most $l$ cyclic subgroups each of order at most $p_0$. Therefore $P$ is of order at most ${p_0}^l$.  Since $p_0$ is a divisor of $q$, the result follows.
\end{proof}



\section{Results on varieties}

In this section we will establish  the main results on varieties, that is, Theorems \ref{variety1} and \ref{variety2}.  


\subsection{The variety  $\frak X(j,q,w)$}
For positive integers $j,q$ and a word $w$, denote by $\frak X= \frak X(j,q,w)$ the class of all groups $G$ with the following properties:\begin{enumerate}
\item $G$ satisfies the law $w^q\equiv1$;

\item  The verbal subgroup $w(G)$ is locally finite;

\item $h(K)\leq j$ for any locally soluble subgroup $K\leq w(G)$.
\end{enumerate}
The main goal of this subsection is to prove that the class $\frak X$ is a variety for any multilinear commutator $w$.

The following result will be useful.
\begin{proposition}\label{cartesian_w} Let $j$ be a positive integer and $w$ a word. Assume $\{G_{\lambda}, \lambda\in \Lambda\}$ is a family of groups such that, for each $\lambda\in \Lambda$, all locally soluble subgroups $K$ of $w(G_{\lambda})$ satisfy $h(K)\leq j$. Let $G$ be the Cartesian product of the groups $\{G_{\lambda}\}$ and assume that $w(G)$ is locally finite. Then $h(H)\leq j$ for any locally soluble subgroup $H\leq w(G)$.
\end{proposition}

\begin{proof}
Since $w(G)$ is locally finite,  $h(K)\leq j$ if and only if $\delta_j^*(K)=1$ whenever $K$ is a locally soluble subgroup of $w(G)$.
By  Lemma \ref{loc_finite_finit_generat} it suffices to show that $h(H)\leq j$ for any finite soluble subgroup $H$ of $w(G)$. Let us prove that $\delta_j^*(H)=1$. Denote by $H_\lambda$ the projection of $H$ on the Cartesian factor $G_\lambda$.
Choose $x\in\delta_j^*(H) $ and observe that the projection $x_\lambda$ of $x$ in $H_{\lambda}$ is trivial because $h(H_\lambda)\leq j$. Hence $x$ is trivial  and the result follows.    \end{proof}

Next we consider the particular case of Theorem \ref{variety1} where $w=\delta_k$ is the derived word for some $k\geq 0$.

\begin{proposition}\label{delta_case}
The class $\frak X=\frak X(j,q,\delta_k)$ is a variety.
\end{proposition}
\begin{proof}
The class $\frak X$ is obviously closed with respect to taking subgroups of its members. 

Let us show that $\frak X$ is closed under taking quotients. Let $G\in\frak X$, and let $N$ be a normal subgroup of $G$.
It is clear that the quotient $G/N$ satisfies the  law $\delta_k^q\equiv1$ and moreover the $k$th term of the derived series of $G/N$  is locally finite since it is a homomorphic image of $G^{(k)}$. We need to check the property that every locally soluble subgroup $KN/N$ of $(G/N)^{(k)}$ satisfies the condition $h(KN/N)\leq j$. In view of Lemma \ref{loc_finite_finit_generat} it is enough to check the property $h(KN/N)\leq j$ for finite soluble subgroups of $(G/N)^{(k)}$. Observe that finite subgroups of $G^{(k)}$  satisfy the hypothesis of Lemma \ref{classY}, so an application of  Lemma \ref{classY}(i) yields the desired result.

It remains to show that if $D=\prod G_i$ is a Cartesian product of groups $G_i\in\frak X$, then $D\in \frak X$. The identity $\delta_k^q\equiv 1$ is obviously satisfied in $D$. In view of Proposition \ref{cartesian_w} it is enough to show that $D^{(k)}$ is locally finite. Let $T$ be any finite set of elements in $D^{(k)}$. There exist finitely many $\delta_k$-values $a_1,\ldots,a_m\in D$ such that $T\subseteq \langle a_1,\ldots,a_m\rangle$. Hence it is sufficient to prove that $H=\langle a_1,\ldots,a_m\rangle$ is finite.  

For any $i$ let $H_i$ be the projection of $H$ on $G_i$. Observe that $H_i$ is finite, since it is finitely generated and contained in $G_i^{(k)}$. It follows that $H$ is residually finite. Moreover it follows from Lemma \ref{classY}(ii) that  $h^*(H_i)$ is $j$-bounded, since any locally soluble subgroup $K$ of $G_i^{(k)}$ satisfies $h(K)\leq j$. Thus by Lemma \ref{basis} each $H_i$ has finite order bounded in terms of $j,k,m,q$ only.  For any $i$  let $K_i$ be the normal subgroup (the kernel of the homomorphism) of $H$ such that $H/K_i$ is isomorphic to $H_i$. Note that the intersection of all $K_i$ is trivial. Since $H_i$ has finite order bounded in terms of $j,k,m,q$, the index of any $K_i$ is bounded in terms of the same parameters. Recall that a finitely generated group has only finitely many normal subgroups of any given finite index (see \cite[Theorem 7.2.9]{mhall}). Since $\cap_i K_i=1$, we conclude that $H$ is finite, as desired. The proof is complete.
 \end{proof}

We are now ready to prove Theorem \ref{variety1}: For positive integers $j,q$ and  any multilinear commutator word $w$,  the class $\frak X= \frak X(j,q,w)$ is a variety.
\begin{proof}[Proof of Theorem \ref{variety1}]
Arguing as in the proof of Proposition \ref{delta_case} it is easy to see that the class $\frak X(j,q,w)$ is closed with respect to taking subgroups  and quotients of its members. 

Thus, we only need to show that if $D$ is a Cartesian product of groups from $\frak X$, then $D$ itself belongs to $\frak X$. The group $D$ obviously satisfies the identity $w^q\equiv1$.  In view of Proposition \ref{cartesian_w} it is sufficient to show that $w(D)$ is locally finite. By Lemma \ref{shum2000_lemma41} there exists a number $t$ such that the identity $\delta_t^q \equiv 1$ holds in any group that belongs to $\frak X(j,q,w)$. Moreover if $G\in \frak X(j,q,w)$, then $G^{(t)}\leq w(G)$ and so any locally soluble subgroup $K$ of $G^{(t)}$ satisfies the property $h(K)\leq j$. It follows from Proposition \ref{delta_case} that $D\in \frak X(j,q,\delta_t)$.  Hence $D^{(t)}$ is locally finite and by applying Lemma \ref{shum2000_lemma42} to the quotient $D/D^{(t)}$ we obtain that $w(D/D^{(t)})$ is locally finite. Hence, $w(D)$ is locally finite as well. 
\end{proof}


\subsection{The variety $\frak Y(j,k,n,q)$}

 For any integers $j,q\geq 1$ and $k,n\geq 0$, let $\delta_k$ be the derived word and recall that $\frak Y=\frak Y(j,k,n,q)$ denotes  the class of all groups $G$ such that:\begin{enumerate}

\item $G$ satisfies the law $[\delta_k^q,{}_n\, y]\equiv1$;

\item The commutator subgroup $G^{(k)}$ is (locally nilpotent)-by-(locally finite);

\item  $h(K)\leq j$ for any locally soluble subgroup $K\leq G^{(k)}$. 
\end{enumerate}

Our goal in this subsection is to show that $\frak Y$ is a variety. 
The next results will be useful later on. 
\begin{theorem}\label{thmB} \cite[Theorem B]{STT2} Let $n$ be a positive integer and $v$ an arbitrary word. Then the class of all groups $G$ in which the $v$-values are right $n$-Engel and $v(G)$ is locally nilpotent is a variety.
\end{theorem}

\begin{lemma}\label{similarto_STT_lemma41}
Let  $m,n\geq 1$ be  positive integers. Let $G$ be a nilpotent group generated by $m$ elements that are right $n$-Engel. Suppose that $G$ satisfies an identity $v\equiv 1$. Then the nilpotency class of $G$ is  $(m,n,v)$-bounded.
\end{lemma}

\begin{proof}
Suppose that the result is false. Then there exists an infinite sequence $\{G_i, i\geq 1\}$ of nilpotent groups satisfying the hypotheses of the lemma such that the nilpotency class of $G_i$ tends to infinity as $i$ does. Each $G_i$ is generated by $m$ elements, say,  $x_{i1},\ldots,x_{im}$, which are right $n$-Engel. Let $C$ be the Cartesian product of the groups $G_i$. For any $i\geq 1$ and $1\leq j \leq m$, let $y_j$ be the element of $C$ whose $i$th component is equal to $x_{ij}$.  Put $H=\langle y_1,\ldots, y_m\rangle$. It is well known that a finitely generated nilpotent group is residually finite (see \cite[Section 5.4]{Rob}). Since $H$ is residually nilpotent and finitely generated, we observe that $H$ is residually finite and satisfies the identity $v\equiv 1$. Moreover, $H$ is generated by $m$ elements which are right $n$-Engel. By \cite[Theorem A]{STT2} the group $H$ is nilpotent, say of class $c$, and so each $G_i$ is nilpotent of class at most $c$, a contradiction. This completes the proof.
\end{proof}

We will require the following  result  taken from \cite{STT1}.
\begin{proposition}\label{STT_prop35}\cite[Proposition 3.5]{STT1} Let $G$ be a group with an ascending normal series whose quotients are either locally soluble or locally finite. Then the Engel elements of $G$ lie in $F(G)$.
\end{proposition}

We remark that the above proposition applies to both left and right Engel elements.

In what follows we will show that the class $\frak Y=\frak Y(j,k,n,q)$ is closed under taking Cartesian products of its members. In the remaining part of this subsection the word $\delta_k$ will be denoted by the symbol  $u$.
\begin{lemma}\label{loc_nilpotent_w^q(G)}
If $G\in\frak Y$, then $u^q(G)$ is locally nilpotent.
\end{lemma}
\begin{proof}
If $G\in\frak Y$, then $G^{(k)}$ is (locally nilpotent)-by-(locally finite) and the $u^q$-values in $G$ are right $n$-Engel elements. By Proposition \ref{STT_prop35} the $u^q$-values in $G$ are in the Hirsch-Plotkin radical, and so $u^q(G)$ is locally nilpotent, as desired.
\end{proof}

\begin{lemma}\label{fitting} Let $G\in\frak Y$ and $K=\langle a_1,\dots,a_s\rangle$, where $a_i$ are $u$-values in $G$. Let $L$ be the subgroup generated by all $u^q$-values of $G$ contained in $K$. Then $L$ is nilpotent of $(u,q,n,s,j)$-bounded class. \end{lemma}
\begin{proof} By Lemma \ref{loc_nilpotent_w^q(G)} $u^q(G)$ is locally nilpotent. Note that $G^{(k)}$ is (locally nilpotent)-by-(locally finite), since $G\in\frak Y$. It follows that $K$ is nilpotent-by-finite. 

Observe that $L$ is normal in $K$. Recall that a finitely generated nilpotent group is residually finite. Therefore $K/L$ is residually finite.  Let $X$ be the set of all $u$-values of $G$ contained in $K$. Observe that $K=\langle X\rangle$ and $X$ is a commutator-closed set.  In view of Lemma \ref{basis}, finite images of $K/L$ have order bounded in terms of $u,q,s,n$ and $j$. Since $K/L$ has only boundedly many normal subgroups of any fixed finite index, we conclude that  $L$ has bounded index in $K$ and so $L$ is $m$-generated, for some bounded number $m$. Thus $L$ is nilpotent, being contained in $u^q(G)$.  Since $K$ is residually finite, observe that  the nilpotency class of $L$ is bounded if and only if so is the nilpotency class of the image $\bar L$ of $L$ in a finite quotient $\bar K$ of $K$.  Therefore it is sufficient to bound the nilpotency class of the Sylow $p$-subgroup of $\bar L$ for a prime divisor $p$ of the order of $\bar L$.

 We can pass to the quotient  $\bar K/O_{p'}(\bar K)$ and assume that $\bar L$ is a $p$-group. Recall that $L$ is  $m$-generator and  a generating set consists  of $u^q$-values  of $G$ contained in $K$. Combining this with the Burnside Basis Theorem \cite[5.3.2]{Rob} we deduce that $\bar L$ is generated by the images of at most $m$ $u^q$-values. Since $u^q$-values in $G$ are right $n$-Engel,   in view of Lemma \ref{similarto_STT_lemma41}, the nilpotency class of $\bar{L}$ is bounded. The result follows.\end{proof}

\begin{lemma}\label{engel} Let $G$, $K$, and $L$ be as in Lemma \ref{fitting}. Choose  $b\in K$. Then $L\langle b\rangle$ is nilpotent of bounded class.
\end{lemma}
\begin{proof} By Lemma \ref{fitting} we know that $L$ has  bounded nilpotency class. In view of Hall's criterion for nilpotency \cite{hall} it is sufficient to bound the nilpotency class of $L\langle b\rangle/L'$. Consider the natural action of $b$ on $L/L'$. Since $G$ satisfies the law $[u^q,{}_n\,y]\equiv1$, it follows that $[L/L',{}_n\,b]=1$. This completes the proof.
\end{proof}

\begin{lemma}\label{bengel} Let $G\in\frak Y$ and $K=\langle a_1,\dots,a_s\rangle$, where $a_i$ are $u$-values in $G$. Let $H\leq K$ and let $b\in F(H)$. Then $b$ is $n_0$-Engel in $H$ for some bounded number $n_0$.
\end{lemma}
\begin{proof}  Let $L$ be the subgroup generated by all $u^q$-values of $G$ contained in $K$. Precisely  as in the proof of Lemma \ref{fitting} we are in a position to apply Lemma \ref{basis} and  so  $L$ has bounded index in $K$. Given $y\in H$, we have $[y,b]\in F(H)$. Observe that the image of $F(H)$ in the quotient $K/L$ is contained in the Fitting  subgroup of $K/L$, and so it has bounded nilpotency class (Obviously  the nilpotency class of $F(K/L)$ is less than the $|K/L|$). Combining this observation  with the fact that $L\langle b\rangle$ has bounded nilpotency class because of Lemma \ref{engel}, we deduce that $b$  is $n_0$-Engel in $H$ for some bounded number $n_0$, as desired.
 \end{proof}

Recall that a group is locally graded if every nontrivial finitely generated subgroup has a proper subgroup of finite index. The class of locally graded groups is quite large and in particular it contains all residually finite groups. The main result in \cite{LMS} states that if $G$ is a locally graded group and $N$ is a locally nilpotent normal subgroup of $G$, then $G/N$ is locally
graded as well. 

Finally we are ready to embark on the proof  of Theorem \ref{variety2}:  for any integers $j,q\geq 1$ and $k,n\geq 0$, the class  of groups  $\frak Y=\frak Y(j,k,n,q)$ is a variety.
\begin{proof}[Proof of Theorem \ref{variety2}]
Obviously, the class $\frak Y=\frak Y(j,k,n,q)$ is closed with respect to taking subgroups of its members. Let us show that $\frak Y$ is closed with respect to taking quotients.  Choose $G\in \frak Y$ and let $H$ be an epimorphic image of $G$.   The law $[u^q,{}_n\,y]\equiv 1$ holds in $G$, and the same is obviously true for $H$. Since $G\in \frak Y$, the subgroup $G^{(k)}$ is (locally nilpotent)-by-(locally finite) and  $H^{(k)}$ has the same structure.  We  also know that every locally soluble subgroup of $G^{(k)}$  has Fitting height at most $ j$. In view of Proposition \ref{closed_quotients} we conclude that whenever $K$ is a locally soluble subgroup of $H^{(k)}$ we have $h(K)\leq j$.

It remains to prove that  $\frak Y$ is closed under  taking Cartesian products. Let $D=\prod G_i$, where $G_i\in\frak Y$.  Obviously the law $[u^q,{}_n\,y]\equiv 1$ is satisfied in $D$. Note that if $G\in\frak Y$, then the $u^q$-values in $G$ are right $n$-Engel and, by Lemma \ref{loc_nilpotent_w^q(G)}, $u^q(G)$ is locally nilpotent. Thus, it follows from Theorem \ref{thmB} that $u^q(D)$ is locally nilpotent too. In order to show that $D^{(k)}$ is (locally nilpotent)-by-(locally finite) we will establish that $D^{(k)}/u^q(D)$ is locally finite.

Take a finite set  $S$ of elements in $D^{(k)}$ and observe that there exist finitely many $u$-values, say $a_1,\dots,a_r$ such that $S$ is contained in $K=\langle a_1,\dots,a_r \rangle$. We need to show that the image of $K$ in $D^{(k)}/u^q(D)$ is finite. Note that the projection of $K$ on every factor of $D$ is nilpotent-by-finite, because it is finitely generated and  contained in $G_i^{(k)}$ that is (locally nilpotent)-by-(locally finite). It follows that $K$ is residually (virtually nilpotent) and hence residually finite as well. Thus, any quotient of $K$ by a locally nilpotent normal subgroup is locally graded because of  the result in \cite{LMS}. Now let $L$ be the subgroup generated  by all $u^q$-values contained in $K$.  In particular the quotient group $T=K/L$ is locally graded. 

Choose any normal subgroup $N$ of finite index in $T$. In view of Lemma \ref{basis} the quotient group $T/N$ is finite of bounded order.  Let $N_0$ be the intersection of all the normal subgroups of finite index in $T$. Suppose that $N_0\neq1$. On the one hand, since $T$ is finitely generated, it has only finitely many subgroup of any given index. On the other hand, the previous argument shows that the index of any normal subgroup of finite index in $T$ is bounded. Combining these two facts we see that $N_0$ has finite index in $T$, too. This implies that $N_0$ is finitely generated and therefore contains a proper subgroup $M$ of finite index since $T$ is locally graded. Denote by $M_0$  the (normal) core of $M$ in $T$. It follows that $M_0$ is a normal subgroup of finite index in $T$ properly contained in  $N_0$. This is in contradiction with the definition of $N_0$, and so $M$ must be trivial.  Thus $T$ is finite and we  conclude that $D^{(k)}/u^q(D)$ is locally finite, as claimed. 

Next we need to show that if $H$ is a locally soluble subgroup of $D^{(k)}$, then we have $h(H)\leq j$. In view of Lemma \ref{enough_fg}  we can assume that $H$ is finitely generated.  Choose a subgroup $U$ generated by finitely many $u$-values such that $H\leq U$.

Observe that the image of $U$ in the quotient $D^{(k)}/u^q(D)$ is finite, since $U$ is finitely generated and $D^{(k)}/u^q(D)$ is locally finite. In what follows we argue as in the proof of Lemma \ref{enough_fg}. Let us denote by $F$ the Hirsch-Plotkin radical of $H$ and  for $t\geq 0$ set $$H_t^*=\{x\in H;\ xF\text{ is a $\delta_t^*$-commutator in $H/F$}\},$$ that is, $H_t^*$ is the set of elements that are $\delta_t^*$-commutators modulo $F$. To  show that $h(H)\leq j$,  it is sufficient to establish that $H_{j-1}^*\subseteq F$. Choose $b\in H_{j-1}^*$.  Let $U_i$ and $H_i$ be the projections of the subgroups $U$ and $H$ on $G_i$. For any $i\geq 1$ denote the projection  of $b$ in $H_i$ by $b_i$. Since the subgroup  $H_i$ is locally soluble and contained in $G_i^{(k)}$, we have $h(H_i)\leq j$ and so $h(H_i/F(H_i))\leq j-1$. Note that the quotient $H_i/F(H_i)$ is locally finite because $H_i \leq G_i^{(k)}$ and $u^q(H_i) \leq F(H_i)$. Similarly to the above, in each $H_i$ we consider subsets  $(H_i)^*_t$ for all $t\geq 0$.  Since any $\delta_{j-1}^*$-commutator in $H_i/F(H_i)$ is trivial, we have  $(H_i)_{j-1}^{*}\subseteq F(H_i)$. Observing that each $H_i$ is a homomorphic image of $H$ and that the image of a $\delta_{j-1}^*$-commutator is again a $\delta_{j-1}^*$-commutator, we see that $b_iF(H_i)$ is a $\delta_{j-1}^*$-commutator. Therefore each $b_i$ belongs to $(H_i)_{j-1}^{*}$ and so $b_i \in F(H_i)$.  It follows from Lemma \ref{bengel} that, for any $i$, the element $b_i$ is $n_0$-Engel in $H_i$ for some  number $n_0$, which does not depend on $i$. Hence $b$ is $n_0$-Engel in $H$ as well, and so $b\in F$, as desired. This implies the result.
\end{proof}



\section{On conciseness of words in residually finite groups} 

In this section we prove the main results on conciseness of words, that is, Theorems \ref{main1} and \ref{main2}. 
\subsection{Generalities}
In what follows we collect some preliminary observations on conciseness of words, which will be needed for the proofs of the main results.

\begin{lemma} \cite[Lemma 3.1]{AS14} \label{AS_lemma31} Let $v$ be any word and $G$ a group such that the set of  $v$-values in $G$  is finite. Then $v(G)$ is finite if and only if it is periodic. \end{lemma}

\begin{lemma}\cite[Lemma 4]{DMS2019}\label{r_abelian} Let $v$ be a word and $G$ a group such that the set of  $v$-values in $G$ is finite with at most $m$ elements. Then the order of $v(G)'$ is $m$-bounded.
\end{lemma}

\begin{lemma} \cite[Corollary 2.7]{AS14} \label{AS_cor}Let $G$ be a soluble-by-finite group, $w$ a multilinear commutator word, and $q\geq 1$ an integer. Suppose that $G$ has only finitely many $w^q$-values. Then $w(G)$ has  finite exponent. 
\end{lemma}

We will also need the following observation.
\begin{lemma}\label{phi}
Let $G$ be a profinite group. Then $\pi(\Phi(G))\subseteq \pi(G/\Phi(G))$.
\end{lemma}
\begin{proof}
Assume that the result is false and let $p\in\pi(\Phi(G))$ while $p\notin \pi(G/\Phi(G))$.  Passing to the quotient $G/O_{p'}(\Phi(G))$ we can also assume that $\Phi(G)$ is a pro-$p$ group while $G/\Phi(G)$ is a pro-$p'$ group. By the profinite analog of the Schur-Zassenhaus Theorem (see \cite[Theorem 2.3.15]{rz}) we have $G=\Phi(G)K$, for some closed subgroup $K$ of $G$ such that $K\cap \Phi(G)=1$ and this yields a contradiction. 
\end{proof}

Note that Lemma \ref{frattini_finito} admits a profinite version: If $N$ is a normal subgroup of a profinite group $G$, then there exists a subgroup $H$ such that $G=HN$ and $H\cap N\leq \Phi(H)$ (see \cite[Lemma 2.8.15]{rz}). This will play an important role in the proof of the next proposition.

Given a word $w$, it is easy to see that the problem on conciseness of $w$ in the class of residually finite groups is equivalent to that in profinite groups. Indeed, suppose that $G$ is a residually finite group in which $w$ takes only finitely many values and note that $G$ naturally embeds into the profinite completion $\hat G$. Observe that $w$ takes only finitely many values in $\hat G$ and moreover $w(G)$ is finite if and only if $w(\hat G)$ is so. Thus, we may work with $\hat G$ rather than $G$.

We will require the following more precise observation.

\begin{proposition}\label{start} 
A word $w$ is concise in residually finite groups if and only if $w$ is concise in groups that are virtually pro-$p$.
\end{proposition}

\begin{proof} Since any profinite group can be regarded as an abstract residually finite group, one implication is obvious. So  assume that the word $w$ is concise in  virtually pro-$p$ groups. In view of the above it is sufficient to show that $w$ is concise in profinite groups. Let $G$ be a profinite group in which $w$ takes only finitely many values, say $\{x_1,\ldots,x_m\}$. We want to show that $w(G)$ is finite. Choose an open normal subgroup $N$ such that $w$ takes exactly $m$ values in $G/N$. This is possible since it is enough to choose $N$ in such a way that it trivially intersects the finite set $\{x_i^{-1}x_j; 1\leq i<j\leq m\}$.  By  the profinite version of Lemma \ref{frattini_finito} there exists a subgroup  $H\leq G$  such that $G=NH$ and $H\cap N\leq\Phi(H)$. It is clear that $w$ takes exactly $m$ values in $G/N$ and therefore also in $H$. Hence, $w(H)=w(G)$. Since $H\cap N\leq\Phi(H)$, it follows that $H$ is virtually pronilpotent.  Moreover $\pi(H/\Phi(H))$ is finite  and so, by Lemma \ref{phi}, $\pi(\Phi(H))$ is finite as well.

Now we  argue by induction on $t=|\pi(\Phi(H))|$. If $t=1$, then $H$ is virtually pro-$p$, for some prime $p$. Thus by hypothesis $w(H)$ is finite, and we are done. Assume that $t\geq2$. If $p\in\pi(\Phi(H))$, observe that $\Phi(H)=P\times Q$, where $P$ is the Sylow $p$-subgroup of $\Phi(H)$ and $Q$ the Hall $p'$-subgroup. By induction both $w(H/P)$ and  $w(H/Q)$ are finite. We conclude that $w(H)$ is finite since it is isomorphic to a subgroup of the direct product of $w(H/P)$ and $w(H/Q)$. The result follows.
\end{proof}

\subsection{The word $w^q$}
In this subsection we will show  that, for any integer $q\geq 1$, the word $w^q$ is concise in residually finite groups, whenever $w$ is a multilinear commutator word.

\begin{proof}[Proof of Theorem \ref{main1}]  In view of Proposition \ref{start} it is sufficient to show that if $G$ is a virtually pro-$p$ group in which $v=w^q$ takes only finitely many values, then $v(G)$ is finite. Taking into account Lemma \ref{r_abelian} we pass to the quotient $G/v(G)'$ and assume that $v(G)$ is abelian. We can also assume, without loss of generality, that $v(G)$ is torsion-free since we can pass to the quotient over the finite normal subgroup generated by the elements of finite order in $v(G)$. Since $G$ is virtually pro-$p$, there is $j\geq 1$ bounding the Fitting height of the prosoluble subgroups of $G$. Indeed, since $G$ has a normal pro-$p$ subgroup of finite index, say $j$, any prosoluble subgroup $K$ of $G$  also has a normal pro-$p$ subgroup, say $P$, of finite index $j$. Then obviously the Fitting height of the finite group $K/P$ is smaller than $j$ and so $h(K)\leq j$, as claimed.

Write $\frak X=\frak X(j,q,w)$ for the variety handled in Theorem \ref{variety1}. Let $H$ be an open normal subgroup of $G$ containing no nontrivial $v$-values of $G$. Taking into account Lemma \ref{classY}(i) one observes that continuous finite images of $H$ belong to $\frak X$. It follows that $H$ is a pro-$\frak X$ group. So $H$ embeds into a Cartesian product of finite $\frak X$-groups. Since $\frak X$ is a variety, we conclude that, viewed as an abstract group, $H\in\frak X$.

From now on we treat $G$ as an abstract group. Observe that the verbal subgroup $w(H)$ is locally finite (because $H\in\frak X$) and so periodic. Pass to the quotient $G/w(H)$.  By Lemma \ref{shum2000_lemma41}  the image of $H$ in the quotient  $G/w(H)$ is soluble.  Thus   $G/w(H)$ is virtually soluble and  Lemma \ref{AS_cor} tells us that  $w(G/w(H))$ has finite exponent. Since any $v$-value in $G$ is an element of $w(G)$, we conclude  that  $v(G)$ is periodic. On the other hand, $v(G)$ is torsion-free. Hence, $v(G)=1$. This concludes the proof.  
\end{proof}


\subsection{The word $[w^q,{}_n\,y]$}
The goal of this subsection is to establish that the word $[w^q,{}_n\,y]$ is concise in residually finite groups, whenever $w$ is a multilinear commutator word.

The next result is immediate from \cite[Proposition 5.4]{DMS2020}.

\begin{proposition}\label{ggd_pro54} Let $n,q\geq1$ and suppose that $w= w(x_1,\ldots,x_k)$ is a multilinear commutator word and $v=[w^q,{}_n\,y]$. Let $G$ be a virtually soluble group in which $v$ takes only finitely many values. Then $v(G)$ is finite.
\end{proposition}

The following lemmas are taken from \cite{DMS2020}. 
\begin{lemma}\cite[Lemma 2.12]{DMS2020}\label{dms_lemma3} Let $n,d,q$ be positive integers. Suppose that $w=w(x_1,\ldots,x_k)$ is a multilinear commutator word and $v=[w^q,{}_n\,y]$. Let $G$ be a residually finite group such that $v(G)$ is abelian. Let $g_1,\ldots, g_d$ be $w^q$-values which are right $(n+1)$-Engel. Then for every $t\in G$ the subgroup $\langle g_1,\ldots,g_d,t\rangle$ is nilpotent of $(d,n,w,q)$-bounded class.
\end{lemma}

\begin{lemma}\cite[Corollary 2.15]{DMS2020}\label{dms_lemma4} Let $m,n,c$ be positive integers and $G$ a nilpotent group of class at most $c$. Let $g\in G$ and denote by $Y$ the set of conjugates of elements of the form $[g,{}_n\,x]$, where $x\in G$.  Assume that $Y$ is finite with at most $m$ elements. Then each element in $Y$ has finite $(c,m)$-bounded order.
\end{lemma}

We are now ready to prove Theorem \ref{main2}: Let $w=w(x_1,\ldots,x_k)$ be a multilinear commutator word. For any integers $q\geq 1$ and $n\geq 0$ the word $v=[w^q,{}_n\,y]$ is concise in residually finite groups.

\begin{proof}[Proof of Theorem \ref{main2}] In view of Proposition \ref{start} it is sufficient to show that if $G$ is a virtually pro-$p$ group in which $v$ takes only finitely many values, then $v(G)$ is finite. Taking into account Lemma \ref{r_abelian} we pass to the quotient $G/v(G)'$ and assume that $v(G)$ is abelian. We can also assume without loss of generality that $v(G)$ is torsion-free since we can pass to the quotient over the finite normal subgroup generated by the elements of finite order in $v(G)$. As was shown in the proof of Theorem \ref{main1}, since $G$ is virtually pro-$p$, there is $j\geq1$ such that the prosoluble subgroups $K\leq G$ have the property $h(K)\leq j$. 

 Since $w$ is a multilinear commutator word on $k$ variables, by Lemma \ref{shum2000_lemma41} any $\delta_k$-value is a $w$-value. Let $\frak Y=\frak Y(j,k,n,q)$ be the variety as in Theorem \ref{variety2}. 

Let $H$ be an open normal subgroup of $G$ containing no nontrivial $v$-values of $G$. In view of Lemma \ref{classY}(i) continuous finite images of $H$ belong to $\frak Y$.  Therefore $H$ is a pro-$\frak Y$ group. Thus $H$ embeds into a Cartesian product of finite $\frak Y$-groups and, since $\frak Y$ is a variety,  the abstract commutator subgroup $H^{(k)}$ is (locally nilpotent)-by-(locally finite). In view of  Lemma \ref{loc_nilpotent_w^q(G)}  $\delta_k^q(H)$ is locally nilpotent.

We claim that the quotient $H^{(k)}/\delta_k^q(H)$  is locally finite. 
Indeed, choose finitely many $\delta_k$-values in $H$ and denote by $K$ the abstract subgroup generated by these elements. Note that $K$ is virtually nilpotent and so residually finite. Now let $L$ be the abstract subgroup generated  by all $\delta_k^q$-values of $G$ contained in $K$. Note that $L$ is locally nilpotent, being contained in $\delta_k^q(H)$. The quotient group $T=K/L$ need not be residually finite but, by the result in \cite{LMS}, is locally graded.  

Similarly to what was done in the proof of Theorem \ref{variety2}, choose any normal subgroup $N$ of finite index in $T$. In view of Lemma \ref{basis} the quotient group $T/N$ is finite of bounded order.  Let $N_0$ be the intersection of all normal subgroups of finite index in $T$. On one hand  $T$  has only finitely many subgroups of any given index. On the other hand, the index of any normal subgroup of finite index in $T$ is bounded. Thus  $N_0$ has finite index in $T$. It follows that $N_0$ is finitely generated. Suppose that $N_0\neq1$. Then $N_0$ contains a proper subgroup $M$ of finite index. Denote by $M_0$  the (normal) core of $M$ in $T$. Since $M_0$ is a normal subgroup of finite index in $T$ properly contained in  $N_0$, we get a contradiction.  Hence $M$ must be trivial and  $T$ is finite. This implies that $H^{(k)}$ is locally finite modulo $\delta_k^q(H)$, as claimed. 

Since $\delta_k^q(H)\leq w^q(H)$, the image of $H^{(k)}$ in $G/w^q(H)$  is locally finite. By Lemma \ref{shum2000_lemma42} the image of $w(H)$ in $G/H^{(k)}w^q(H)$ is locally finite. Therefore $w(H)$ is locally finite modulo $w^q(H)$. 

Since $H^{(k)}$ is (locally nilpotent)-by-(locally finite), the group $H$ is (locally nilpotent)-by-(locally finite)-by-soluble. Of course, $H$ satisfies the law $v\equiv 1$,  and so all $w^q$-values in $H$ are right $n$-Engel elements. By Proposition  \ref{STT_prop35} the Engel elements of $H$ form a locally nilpotent subgroup. In particular all $w^q$-values in $H$ belong to the Hirsch-Plotkin radical of $H$ and so $w^q(H)$ is locally nilpotent. Thus $w(H)$ is (locally nilpotent)-by-(locally finite).

 Note that  $H/w(H)$ is a soluble normal subgroup of finite index in $G/w(H)$. It follows from Proposition \ref{ggd_pro54} that $v(G/w(H))$ is finite. Thus, $v(G)\cap w(H)$ has finite index in $v(G)$. The local finiteness of $w(H)/w^q(H)$ combined with the fact that $v(G)$ is finitely generated implies that $v(G)\cap w^q(H)$ has finite index in $v(G)\cap w(H)$. Hence, $v(G)\cap w^q(H)$ has finite index in $v(G)$, and therefore  $v(G)\cap w^q(H)$ is finitely generated.

Choose $t\in G$.  It follows that the  subgroup $\langle w^q(H),t\rangle$ is locally nilpotent. Indeed, as shown above $w^q(H)$ is locally nilpotent and if $x$ is a $w^q$-value in $H$, then for any $y\in G$ we have $[x,_ny]=1$ because $H$ does not contain nontrivial $v$-values. Furthermore by \cite[Exercise 12.3.6]{Rob} the subgroup $\langle x\rangle^{\langle y\rangle}$ is generated by finitely many conjugates of $x$. In order to prove that $w^q(H)\langle t\rangle$ is locally nilpotent, choose a finite subset $X\subseteq w^q(H)\langle t\rangle$.  We need to show that $\langle X\rangle$ is nilpotent. If $X\subseteq w^q(H)$, then $\langle X\rangle$ is nilpotent because $w^q(H)$ is locally nilpotent. Assume now that $X\not\subseteq w^q(H)$. Then there is a finite subset $Y\subseteq w^q(H)$ and a positive integer $i$ such that $\langle X\rangle=\langle Y, t^i\rangle$. So it is sufficient to prove  that $\langle Y,t\rangle$ is nilpotent. Note that $Y$ is contained in a subgroup generated by finitely many $w^q$-values in $H$. Without loss of generality we can assume that $Y=\langle b_1,\dots,b_d\rangle$, where $b_1,\dots,b_d$ are $w^q$-values in $H$. The above observation on the subgroups $\langle x\rangle^{\langle y\rangle}$  implies that $t$ normalizes a subgroup, say $Y_0$, containing $Y$ which is generated by finitely many $w^q$-values in $H$. Since $Y_0$ is nilpotent, we are now in a position to apply Lemma 2.6 and deduce that $\langle Y,t\rangle$ is nilpotent. Hence $\langle w^q(H),t\rangle$ is locally nilpotent, as required.
Combining this with the fact that $v(G)\cap w^q(H)$ is finitely generated, we deduce that  there exists an integer $s$ such that $[v(G)\cap w^q(H),{}_{s}\,t]=1$.

Since $v$ takes only finitely many values in $G$, the index of $C_G(v(G))$ in $G$ is finite and so the conjugation by $t$ induces an automorphism of $v(G)$ of finite order, say $r$. By Lemma  \ref{dms_lemma2} we have $[(v(G)\cap w^q(H))^{r^{s-1}},t]=1$.  As $v(G)\cap w^q(H)$ is abelian, it follows that $[v(G)\cap w^q(H), t]$ has finite exponent dividing $r^{s-1}$. Hence $[v(G)\cap w^q(H), t]=1$, because  $v(G)$ is torsion-free.

Combining   the previous observation with the fact that the conjugation by $t^{r}$ induces the trivial automorphism on $v(G)$, we get that $(v(G)\cap w^q(H))\langle t^{r}\rangle$  is a central subgroup of finite index in $v(G)\langle t\rangle$. Then by Schur's theorem the commutator subgroup of $v(G)\langle t\rangle$ is finite. Obviously the quotient $v(G)\langle t\rangle/v(G)$ is abelian. Hence the commutator subgroup of $v(G)\langle t\rangle$ is contained in $v(G)$ and, taking into account that  $v(G)$ is torsion-free, we have $[v(G),t]=1$.  Since $t$ was chosen in $G$ arbitrarily, we now conclude that $v(G)$ is central in $G$. In particular, if  we take any $w$-value $g$ in $G$ and $t\in G$, then we have $[g^q,{}_{n+1}\,t]=1$.

In view of Lemma \ref{dms_lemma3} the subgroup $\langle g^q,t \rangle$ is nilpotent and so by Lemma \ref{dms_lemma4} we obtain that any element of the form $[g^q,{}_n\,t]$ has finite order. Since $v(G)$ is torsion-free and $[g^q,{}_n\,t]$ is an arbitrary $v$-value, we deduce that $v(G)=1$. This concludes  the proof.
\end{proof}


\end{document}